\documentclass[11pt,a4paper]{amsart}

\usepackage[utf8]{inputenc}

\usepackage{amssymb, graphicx}

\usepackage{amsfonts}
\usepackage{amsmath}
\usepackage{latexsym}

\allowdisplaybreaks[2]

\newtheorem{theorem}{Theorem}[section]
\newtheorem{proposition}[theorem]{Proposition}
\newtheorem{lemma}[theorem]{Lemma}
\newtheorem{corollary}[theorem]{Corollary}

\theoremstyle{remark}

\numberwithin{equation}{section}

\begin{document}

\title[Zero Lie product determined Banach algebras]{Zero Lie product determined Banach algebras}

\author{J. Alaminos}
\author{M. Bre\v sar}
\author{J. Extremera}
\author{A.\,R. Villena}
\address{J. Alaminos, J. Extremera  and A.\,R. Villena, Departamento de An\' alisis
Matem\' atico, Fa\-cul\-tad de Ciencias, Universidad de Granada,
18071 Granada, Spain} \email{alaminos@ugr.es, jlizana@ugr.es,
avillena@ugr.es}
\address{M. Bre\v sar,  Faculty of Mathematics and Physics,  University of Ljubljana,
Jadranska 19, 1000 Ljubljana,
 and Faculty of Natural Sciences and Mathematics, University
of Maribor, Koro\v ska 160, 2000 Maribor, Slovenia} \email{matej.bresar@fmf.uni-lj.si}

\thanks{\emph{Mathematics Subject Classification}. 43A20,  46H05, 46L05,  47B48.}
\keywords{Zero Lie product determined Banach algebra, property $\mathbb{B}$,
 amenability, group algebra, C*-algebra.}
\thanks{
The authors were supported by MINECO grant MTM2015--65020--P.
The first, the third  and the fourth named authors were supported
by  Junta de Andaluc\'{\i}a grant FQM--185.
The second  named author was supported by ARRS grant P1--0288.}

\begin{abstract}
A Banach algebra $A$ is said to be zero Lie product determined if
every continuous bilinear functional
$\varphi \colon A\times A\to \mathbb{C}$ with the property that
$\varphi(a,b)=0$ whenever $a$ and $b$
commute is of the form $\varphi(a,b)=\tau(ab-ba)$ for some $\tau\in A^*$. In the first part of the paper we give some general remarks on this class of algebras. In the second part we consider amenable Banach algebras and show that
 all group algebras
$L^1(G)$ with $G$ an amenable locally compact group are  zero Lie product determined.
\end{abstract}

\maketitle

\section{Introduction}

Let $A$ be a Banach algebra and let $\varphi \colon A\times A\to\mathbb{C}$ be a continuous bilinear functional satisfying
\begin{equation}\label{B}
a,b\in A, \ [a,b]=0 \ \ \Longrightarrow \ \ \varphi(a,b)=0
\end{equation}
(here and subsequently, $[a,b]$ stands for the commutator $ab - ba$).
This is certainly fulfilled if $\varphi$ is of the form
\begin{equation}\label{Bl}
\varphi(a,b)=\tau([a,b])  \quad (a,b\in A)
\end{equation}
for some $\tau$ in $A^*$, the dual of $A$. We will say that $A$ is a
\emph{zero Lie product determined Banach algebra} if, for every continuous bilinear functional
$\varphi \colon A\times A\to\mathbb{C}$ satisfying~\eqref{B}, there exists $\tau\in A^*$ such that~\eqref{Bl} holds.
This is an analytic analogue of the purely algebraic notion of a zero Lie product determined algebra, first indirectly  considered in~\cite{BrSe}
and, slightly later, more systematically
 in~\cite{BGS} (see also   subsequent papers~\cite{Gr,WCZ}).
Further, the concept of a zero Lie product determined Banach algebra can be seen as the Lie version of the notion of a Banach algebra having property $\mathbb{B}$ (see~\cite{ABEV0}), which will also play an important role in this paper.
Another motivation for us for studying this concept is the similarity with the group-theoretic notion of  triviality of Bogomolov multiplier (see, e.g.,~\cite{M}), which made us particularly interested in considering it in the context of group algebras.

The paper is organized as follows. In Section~\ref{s0} we provide  motivating examples. Firstly, by applying a  result by Goldstein~\cite{G} we show that
$C^*$-algebras are
 zero Lie product determined Banach algebras. Secondly, we find a  Banach algebra, even a finite dimensional one, that is not
 zero Lie product determined.
 In Section~\ref{sec1} we prove that in the definition of a zero Lie product determined Banach algebra one can replace the role of $\mathbb{C}$ by any Banach space.
The bulk of the paper is Section~\ref{sec2} in which we
show  that the group algebra $L^1(G)$ of any amenable locally compact group $G$ is a zero Lie product determined Banach algebra.
We actually obtain this as a byproduct of the result concerning the condition
\begin{equation}\label{B1lz}
a,b\in A, \ ab=ba=0 \ \ \Longrightarrow \ \ \varphi(a,b)=0,
\end{equation}
where $A$ is an
amenable Banach algebras with property $\mathbb B$. We remark that~\eqref{B1lz} has also been already  studied in the literature,
but definitive results were  so far obtained only for finite dimensional algebras~\cite{ABEVmat, KLZ}.

\section{Examples}\label{s0}

The goal of this section is to provide examples indicating the nontriviality of the
concept of a  zero Lie product determined Banach algebra.

\begin{proposition} Every $C^*$-algebra is a zero Lie product determined Banach algebra.
\end{proposition}

\begin{proof}
Let $A$ be a $C^*$-algebra, and let $\varphi\colon A\times A\to\mathbb{C}$ be a continuous bilinear functional satisfying~\eqref{B}.
Then the map $\psi\colon A\times A\to\mathbb{C}$ defined by $\psi(a,b)=\varphi(a,b^*)$ for all $a$, $b\in A$ is a continuous sesquilinear functional.
Further, if $a,b\in A$ are self-adjoint and $ab=0$, then $ba=0$, which in turn implies that $[a,b]=0$ and therefore $\psi(a,b)=\varphi(a,b)=0$.
This shows that $\psi$ is orthogonal in the sense of~\cite{G}
(see~\cite[Definition~1.1]{G}).
By~\cite[Theorem~1.10]{G}, $A$ is $\mathbb{C}$-stationary, which means (\cite[Definition~1.5]{G}) that there exist $\tau_1,\tau_2\in A^*$ such that
$\psi(a,b)=\tau_1(ab^*)+\tau_2(b^*a)$ for all $a,b\in A$. Consequently, we have
\begin{equation}\label{ee1}
\varphi(a,b)=\tau_1(ab)+\tau_2(ba) \quad (a,b\in A).
\end{equation}
On the other hand, if $a\in A$, then $[a,a]=0$ and therefore $\varphi(a,a)=0$. Hence $\varphi$ is skew-symmetric and taking into account~\eqref{ee1} we get
\begin{equation}\label{ee2}
\varphi(a,b)=-\varphi(b,a)=-\tau_1(ba)-\tau_2(ab) \quad (a,b\in A).
\end{equation}
Adding~\eqref{ee1} and~\eqref{ee2}, we obtain
\[
2\varphi(a,b)=\tau_1([a,b])-\tau_2([a,b]) \quad (a,b\in A),
\]
which shows that $\varphi$ is of the form~\eqref{Bl}, where $\tau\in A^*$ is defined by $\tau=\tfrac{1}{2}(\tau_1-\tau_2)$.
\end{proof}

We will now give an example of  a finite dimensional Banach algebra  that is not  zero Lie product determined.
This is of interest also from a purely algebraic viewpoint. Namely, so far only examples of infinite dimensional  algebras that are not
zero Lie product determined were found~\cite{BGS} (since bilinear functionals are automatically continuous  in finite dimension, in this framework
there is no difference between ``zero Lie product determined Banach algebra'' and ``zero Lie product  determined  algebra'').
The algebra from the next proposition can be thought of as the  Grassmann algebra with four generators to which we add another relation.

\begin{proposition}
The 10-dimensional Banach algebra
\[
A= \mathbb{C}\big\langle x_1,x_2,x_3,x_4\,|\, x_1 x_2 = x_3 x_4, x_i^2=0, x_i x_j=-x_j x_i,\, i,j=1,2,3,4\big\rangle
\]
is not  zero Lie product determined.
\end{proposition}

\begin{proof}
It is easy to check that the elements
\[
1,\, x_1, \, x_2,\, x_3,\, x_4,\, x_1x_2, \, x_1x_3, \, x_1x_4,\, x_2x_3,\,x_2x_4
\]
form a basis of $A$ (so that $\dim_\mathbb{C} A = 10$). Note that $1$ and all $x_i x_j$ lie in $Z$, the center  of $A$.

Define a bilinear functional $\varphi \colon A\times A\to\mathbb{C}$ by
\[
\varphi(x_1,x_2) = -\varphi(x_2,x_1) = 1
\]
and
\[
\varphi(u,v) = 0
\]
for all other pairs of elements from our basis. Take a pair of commuting elements $a$, $b\in A$. We can write
\[
a = \sum_{i=1}^4 \lambda_i x_i  + z\quad\mbox{and}\quad b= \sum_{j=1}^4 \mu_j x_j +  w,
\]
where $\lambda_i,\mu_j\in \mathbb{C}$ and $z,w\in Z$. Our goal is to show that $\varphi(a,b)=\lambda_1\mu_2-\lambda_2\mu_1$ is $0$.
From $[a,b]=0$ we  obtain
\[
\Big[\sum_{i=1}^4 \lambda_i x_i, \sum_{j=1}^4 \mu_j x_j\Big] =0,
\]
which yields
\begin{align*}
&\big((\lambda_1\mu_2 - \lambda_2\mu_1)+ (\lambda_3\mu_4 - \lambda_4\mu_3)\big)x_1x_2\\
+& (\lambda_1\mu_3 - \lambda_3\mu_1)x_1x_3 + (\lambda_2\mu_3 - \lambda_3\mu_2)x_2x_3 \\
+&  (\lambda_1\mu_4 - \lambda_4\mu_1)x_1x_4+ (\lambda_2\mu_4 - \lambda_4\mu_2)x_2x_4 =0.
\end{align*}
Consequently,
\begin{equation} \label{a}
(\lambda_1\mu_2 - \lambda_2\mu_1)+ (\lambda_3\mu_4 - \lambda_4\mu_3) =0,
\end{equation}
\begin{equation} \label{bb}
\lambda_1\mu_3 = \lambda_3\mu_1,\,\,\,\lambda_2\mu_3 = \lambda_3\mu_2,
\end{equation}
\begin{equation} \label{c}
\lambda_1\mu_4 = \lambda_4\mu_1,\,\,\,\lambda_2\mu_4 = \lambda_4\mu_2.
\end{equation}
Note that~\eqref{bb} yields
\[
(\lambda_1\mu_2 - \lambda_2\mu_1)\mu_3=0.
\]
and, similarly,~\eqref{c} yields
\[
(\lambda_1\mu_2 - \lambda_2\mu_1)\mu_4=0.
\]
But then we infer from~\eqref{a} that $\lambda_1\mu_2 - \lambda_2\mu_1=0$, as desired. We have thereby proved that $\varphi$ satisfies~\eqref{B}.
However, since
$\varphi(x_1,x_2)\ne \varphi(x_3,x_4)$, we see from $[x_1,x_2] = [x_3,x_4]$ that $\varphi$ does not satisfy~\eqref{Bl}.
\end{proof}

\section{An alternative definition}\label{sec1}

From now on, we write $[A,A]$ for the linear span of all commutators of the Banach algebra $A$.

\begin{proposition}
Let $A$ be a Banach algebra. Then the following properties are equivalent:
\begin{enumerate}
\item
the algebra $A$ is a zero Lie product determined Banach algebra,
\item
for each Banach space $X$, every continuous bilinear map $\varphi\colon A\times A\to X$ with the property
that $\varphi(a,b)=0$ whenever $a,b\in A$ are such that $[a,b]=0$
is of the form $\varphi(a,b)=T([a,b])$ $(a,b\in A)$ for
a unique continuous linear map $T\colon [A,A]\to X$.
\end{enumerate}
\end{proposition}

\begin{proof}
Suppose that (1) holds.
Let $X$ be a Banach space and let $\varphi\colon A\times A\to X$ be a continuous bilinear map
with the property that $\varphi(a,b)=0$ whenever $a,b\in A$ are such that $[a,b]=0$.
For each $\xi\in X^*$, the continuous bilinear functional $\xi\circ\varphi\colon A\times A\to\mathbb{C}$
satisfies~\eqref{B}. Therefore there exists a unique $\tau(\xi)\in [A,A]^{*}$ such that
$\xi(\varphi(a,b))=\tau(\xi)([a,b])$ for all $a,b\in A$. It is clear that the map $\tau\colon X^*\to [A,A]^*$
is linear. We next show that $\tau$ is continuous. Let $(\xi_n)$ be a sequence in $X^*$ with $\lim\xi_n=0$
and $\lim\tau(\xi_n)=\xi$ for some $\xi\in [A,A]^*$. For each $a,b\in A$, we have
\[
0=\lim\xi_n(\varphi(a,b))=\lim\tau(\xi_n)([a,b])=\xi([a,b]).
\]
We thus have $\xi=0$, and the closed graph theorem yields the continuity of $\tau$.

For all $a_1,\ldots,a_n,b_1,\ldots,b_n\in A$ and $\xi\in X^*$ we have
\begin{equation}\label{11667}
\begin{aligned}
\xi\Bigl(\sum_{k=1}^n\varphi(a_k,b_k)\Bigr) & =
\sum_{k=1}^n\xi\bigl(\varphi(a_k,b_k)\bigr) =
\sum_{k=1}^n\tau(\xi)([a_k,b_k]) \\
& = \tau(\xi)\Bigl(\sum_{k=1}^n[a_k,b_k]\Bigr).
\end{aligned}
\end{equation}
Consequently, if $a_1,\ldots,a_n,b_1,\ldots,b_n\in A$ are such that $\sum_{k=1}^n[a_k,b_k]=0$,
then $\xi\bigl(\sum_{k=1}^n\varphi(a_k,b_k)\bigr)=0$ for each $\xi\in X^*$, and hence
$\sum_{k=1}^n\varphi(a_k,b_k)=0$. We thus can define a linear map $T\colon [A,A]\to X$ by
\[
T\Bigl(\sum_{k=1}^n[a_k,b_k]\Bigr)=\sum_{k=1}^n\varphi(a_k,b_k)
\]
for all $a_1,\ldots,a_n,b_1,\ldots,b_n\in A$. Of  course, $\varphi(a,b)=T([a,b])$ for all $a,b\in A$.
Our next concern is the continuity of $T$.
Let  $a_1,\ldots,a_n,b_1,\ldots,b_n\in A$. Then there exists $\xi\in X^*$ such that
\[
\xi\Bigl(\sum_{k=1}^n\varphi(a_k,b_k)\Bigr)=\Bigl\Vert\sum_{k=1}^n\varphi(a_k,b_k)\Bigr\Vert.
\]
On account of~\eqref{11667}, we have
\begin{align*}
\Bigl\Vert T\Bigl(\sum_{k=1}^n[a_k,b_k]\Bigr)\Bigr\Vert & =
\Bigl\Vert\sum_{k=1}^n\varphi(a_k,b_k)\Bigr\Vert  = \xi\Bigl(\sum_{k=1}^n\varphi(a_k,b_k)\Bigr) \\
& = \Bigl\vert\tau(\xi)\Bigl(\sum_{k=1}^n[a_k,b_k]\Bigr)\Bigr\vert
 \le \Vert\tau(\xi)\Vert\Bigl\Vert\sum_{k=1}^n[a_k,b_k]\Bigr\Vert \\
& \le \Vert\tau\Vert\Bigl\Vert\sum_{k=1}^n[a_k,b_k]\Bigr\Vert,
\end{align*}
which shows the continuity of $T$, and hence that property (2) holds.

We now assume that (2) holds. Let $\varphi\colon A\times A\to\mathbb{C}$ be a continuous bilinear functional  satisfying~\eqref{B}.
By applying property (2) with $X=\mathbb{C}$, we get  $\tau\in [A,A]^*$ such that $\varphi(a,b)=\tau([a,b])$ $(a,b\in A)$. The functional
$\tau$ can be extended to a continuous linear functional on $A$ so that (1) is obtained.
\end{proof}

\section{Amenable Banach algebras with property $\mathbb{B}$}\label{sec2}

We say that a Banach algebra $A$ has \emph{property $\mathbb{B}$} if for every continuous bilinear functional
$\varphi \colon A\times A\to \mathbb{C}$, the condition
\begin{equation}\label{B1}
a,b\in A, \ ab=0 \ \ \Rightarrow \ \ \varphi(a,b)=0
\end{equation}
implies the condition
\begin{equation}\label{B11}
\varphi(ab,c)=\varphi(a,bc)  \quad (a,b,c\in A).
\end{equation}
According to~\cite[Remark 2.1]{AFA}, this definition agrees with the one given in the seminal paper~\cite{ABEV0},
i.e., the Banach algebra $A$ has property $\mathbb{B}$ if and only if for each Banach space $X$ and for each
continuous bilinear map $\varphi\colon A\times A\to X$ the condition~\eqref{B1} implies the condition~\eqref{B11}.
We remark that if $A$ has a bounded approximate identity,~\eqref{B11} is equivalent to the condition that
$\varphi(a,b)= \tau(ab)$ for some  $\tau\in A^*$ (see~\cite[Lemma~2.3]{ABEV0}).
In~\cite{ABEV0} it was shown that many important examples of Banach algebras,
including  $C^*$-algebras,  group algebras on arbitrary locally compact groups,
and the algebra $\mathcal{A}(X)$ of all approximable operators on any
Banach space $X$, have property $\mathbb{B}$,
and that this property can be applied to a variety of problems. Since then, a number of papers treating
property $\mathbb{B}$ have been published; see the last paper in the series~\cite{ABESV2} and references therein.

The class of amenable Banach algebras is of great significance.
We refer the reader to~\cite{R} for the necessary background on amenability.
There are different characterizations of amenable Banach algebras. The seminal one comes from B. E. Johnson: vanishing of a
certain cohomology group.
For our purposes here, the best way to introduce the amenability is the following.
Let $A$ be a Banach algebra.
The projective tensor product $A\widehat{\otimes}A$ becomes a Banach
$A$-bimodule for the products defined by
\[
a\cdot(b\otimes c)=(ab)\otimes c
\]
and
\[
(b\otimes c)\cdot a=b\otimes (ca)
\]
for all $a,b,c\in A$.
There is a unique continuous linear map $\pi\colon A\widehat{\otimes}A\to A$ such
that
\[
\pi(a\otimes b)=ab
\]
for all $a,b\in A$.
The map $\pi$ is the projective induced product map, and it is an $A$-bimodule homomorphism.
An \emph{approximate diagonal} for $A$ is a bounded net $(u_{\lambda})_{\lambda\in\Lambda}$
in $A\widehat{\otimes}A$ such that, for each $a\in A$, we have
\begin{equation}\label{ad1}
\lim_{\lambda\in\Lambda}(a\cdot u_ \lambda-u_\lambda\cdot a)=0
\end{equation}
and
\begin{equation}\label{ad2}
\lim_{\lambda\in\Lambda}\pi(u_\lambda)a=a.
\end{equation}
We point out that~\eqref{ad1} together with~\eqref{ad2} implies that also $\lim a\pi(u_\lambda)=a$ for each $a\in A$.
Consequently, the net $(\pi(u_\lambda))_{\lambda\in\Lambda}$ is a bounded approximate identity for $A$.
The Banach algebra $A$ is \emph{amenable} if and only if $A$ has an approximate diagonal.

Throughout this section we are notably interested in amenable Banach algebras having property $\mathbb{B}$.
According to~\cite{R}, the following are examples of amenable Banach algebras (which we already know to have property $\mathbb{B}$):
nuclear $C^*$-algebras, the group algebra $L^1(G)$ for each amenable locally compact group $G$,
and the algebra $\mathcal{A}(X)$ for Banach spaces with certain approximation properties
(this includes the Banach space $C_0(\Omega)$ for each locally compact Hausdorff space $\Omega$ and
the Banach space $L^p(\mu)$ for each measure space $(\Omega,\Sigma,\mu)$ and each $p\in[1,\infty]$).

We begin with a lemma whose version  appears also in~\cite{ABEV}.

\begin{lemma}\label{l1519}
Let $A$ be a Banach algebra with property $\mathbb{B}$ and having a bounded approximate identity,
let $X$ be a Banach space,  and
let $\varphi\colon A\times A\to X$ be a continuous bilinear map satisfying the condition:
\begin{equation*}
a,b\in A, \ ab=ba=0 \ \Rightarrow \ \varphi(a,b)=0.
\end{equation*}
Then
\begin{equation}\label{pat3}
\varphi(ab,cd)-\varphi(a,bcd)+\varphi(da,bc)-\varphi(dab,c)=0 \quad (a,b,c,d\in A)
\end{equation}
and there exists a continuous linear operator $S\colon A\to X$
such that
\begin{equation}\label{eqfifi}
\varphi(ab,c)-\varphi(b,ca)+\varphi(bc,a)=S(abc) \quad (a,b,c\in A).
\end{equation}
\end{lemma}

\begin{proof}
Let  $\mathcal{B}^2(A;X)$ denote the Banach space of all continuous bilinear maps from
$A\times A$ to $X$, and let $\mathcal{B}_0^2(A;X)$ denote the closed subspace of
$\mathcal{B}^2(A;X)$ consisting of those bilinear maps $\varphi$ which
satisfy~\eqref{B1}. We define
\[
\psi\colon A\times A\to\mathcal{B}^2(A;X)
\]
by
\[
\psi(a,b)(s,t)=\varphi(bs,ta) \quad (a,b,s,t\in A).
\]
It is immediate to check that $\psi(a,b)\in\mathcal{B}_0^2(A;X)$
whenever $a,b\in A$ are such that $ab=0$. Consequently, the
continuous bilinear map
\[
\widetilde{\psi}\colon A\times A\to \mathcal{B}^2(A;X)/\mathcal{B}_0^2(A;X)
\]
defined by
\[
\widetilde{\psi}(a,b)=\psi(a,b)+\mathcal{B}_0^2(A;X) \quad (a,b\in A)
\]
satisfies~\eqref{B1}.
Property $\mathbb{B}$ then gives
\begin{equation*}
\psi(ab,c)-\psi(a,bc)\in\mathcal{B}_0^2(A;X) \quad (a,b,c\in A).
\end{equation*}
For each $a,b,c\in A$, property $\mathbb{B}$ now yields
\[
\bigl(\psi(ab,c)-\psi(a,bc)\bigr)(rs,t)=
\bigl(\psi(ab,c)-\psi(a,bc)\bigr)(r,st)
\]
for all $r$, $s$, $t\in A$.
Hence
\begin{equation}\label{pat2}
\varphi(crs,tab)-\varphi(bcrs,ta)-\varphi(cr,stab)+\varphi(bcr,sta)=0
\end{equation}
for all $a$, $b$, $c$, $r$, $s$, $t\in A$.

Let $(\rho_\lambda)_{\lambda\in\Lambda}$ be an approximate identity of
$A$ of bound $C$. For each $a$, $b$, $c$, $r$, $s\in A$,
we apply~\eqref{pat2} with the element $t$ replaced by $\rho_\lambda$ ($\lambda\in\Lambda$)
and then we take the limit to arrive at
\begin{equation}\label{pat2b}
\varphi(crs,ab)-\varphi(bcrs,a)-\varphi(cr,sab)+\varphi(bcr,sa)=0.
\end{equation}
We now replace $r$ by $\rho_\lambda$ ($\lambda\in\Lambda$) in~\eqref{pat2b}
and take the limit to get
\begin{equation*}
\varphi(cs,ab)-\varphi(bcs,a)-\varphi(c,sab)+\varphi(bc,sa)=0,
\end{equation*}
which gives~\eqref{pat3}.

By applying~\eqref{pat3} with the element $c$ replaced by $\rho_\lambda$ ($\lambda\in\Lambda$)
we see that the net $(dab, \varphi(\rho_\lambda))_{\lambda\in\Lambda}$ is convergent and by taking the limit in~\eqref{pat3} we arrive at
\begin{equation}\label{e2148}
\begin{aligned}
&\varphi(ab,d)-\varphi(a,bd)+\varphi(da,b)-
\lim_{\lambda\in\Lambda}\varphi(dab,\rho_\lambda)
\\
=& \lim_{\lambda\in\Lambda}
\bigl(\varphi(ab,\rho_\lambda d)-\varphi(a,b\rho_\lambda d)+ \varphi(da,b\rho_\lambda)-\varphi(dab,\rho_\lambda)\bigr)=0
\end{aligned}
\end{equation}
for all $a$, $b$, $d\in A$.
By Cohen's factorization theorem (see~\cite[Corollary~11 in \S 11]{bd}),
each $c\in A$ can be written in the form $c=dab$ with
$a$, $b$, $d\in A$, and
hence the net $(\varphi(c,\rho_\lambda))_{\lambda\in\Lambda}$ is convergent.
We can thus  define a linear operator
$S\colon A\to X$
by
\[
S(a)=\lim_{\lambda\in\Lambda}\varphi(a,\rho_\lambda)
\]
for each $a\in A$.
Since
$\Vert\varphi(a,\rho_\lambda)\Vert\le C \Vert\varphi\Vert \Vert a\Vert$
for all $a\in A$ and $\lambda\in\Lambda$, it follows that
$\Vert S(a)\Vert\le C \Vert\varphi\Vert \Vert a\Vert$ for each $a\in A$,
which implies that $S$ is continuous. Further,~\eqref{e2148}
gives~\eqref{eqfifi}.
\end{proof}

\begin{lemma}\label{l1520}
Let $A$ be an amenable Banach algebra,
let $X$ be a Banach space, and
let $\varphi\colon A\times A\to X$ be a continuous bilinear map.
Suppose that there exists a continuous linear operator $S\colon A\to X$
such that
\begin{equation}\label{b}
\varphi(ab,c)-\varphi(b,ca)+\varphi(bc,a)=S(abc) \quad (a,b,c\in A).
\end{equation}
Then there exist continuous linear operators
$\Phi\colon [A,A]\to X$ and $\Psi\colon A\to X$ such that
\begin{equation*}
\varphi(a,b)=\Phi([a,b])+\Psi(a\circ b) \quad (a,b\in A).
\end{equation*}
Here and subsequently, $a\circ b$ stands for $ab + ba$.
\end{lemma}

\begin{proof}
Let $(u_\lambda)_{\lambda\in\Lambda}$ be an approximate diagonal for $A$ of bound $C$,
and let $\mathcal{U}$ be an ultrafilter on $\Lambda$ refining the order filter.
On  account of the Banach-Alaoglu theorem, each bounded subset of the bidual $X^{**}$
of $X$ is relatively compact with respect to the weak$^*$-topology. Consequently,
each bounded net $(x_\lambda)_{\lambda\in\Lambda}$ in $X$ has a unique
limit in $X^{**}$ with respect to the weak$^*$-topology along the ultrafilter $\mathcal{U}$,
and we write $\displaystyle{\lim_\mathcal{U}} x_\lambda$ for this limit.

Let $\widehat{\varphi}\colon A\widehat{\otimes}A\to X$ be the
unique continuous linear map
such that
\[
\widehat{\varphi}(a\otimes b)=\varphi(a,b)
\]
for all $a,b\in A$.
We define $T\colon A\to X^{**}$ by
\[
T(a)=\lim_{\mathcal{U}}\widehat{\varphi}( u_\lambda\cdot a)
\]
for each $a\in A$.
For each $a\in A$, we have
\begin{equation}\label{e940}
\Vert\widehat{\varphi}(u_\lambda\cdot a)\Vert\le
\Vert\widehat{\varphi}\Vert\Vert u_\lambda\Vert\Vert a\Vert\le C
\Vert\varphi\Vert \Vert a\Vert \quad (\lambda\in\Lambda).
\end{equation}
Hence the net $(\widehat{\varphi}( u_\lambda\cdot a))_{\lambda\in\Lambda}$ is bounded
and the map $T$ is well-defined.
The linearity of the limit along an ultrafilter on a topological linear space gives the linearity of $T$.
Further, from~\eqref{e940} we deduce that $\Vert T(a)\Vert\le C \Vert\varphi\Vert \Vert a \Vert $ for each $a\in A$,
which gives the continuity of $T$.

We now claim that
\begin{equation}\label{e1049b}
\widehat{\varphi}(u\cdot a)=
\widehat{\varphi}(a\cdot u)+\widehat{\varphi}(\pi(u)\otimes a)-S(a\pi(u))
\end{equation}
for all $a\in A$ and $u\in A\widehat{\otimes}A$.
Of course, it suffices to prove~\eqref{e1049b} for the simple tensor products $u=b\otimes c$
with $b,c\in A$. Observe that
\eqref{b} can be written as
\[
\widehat{\varphi}(a\cdot(b\otimes c))-
\widehat{\varphi}((b\otimes c)\cdot a)+
\widehat{\varphi}(\pi(b\otimes c)\otimes a)=S(a\pi(b\otimes c))
\]
and this gives~\eqref{e1049b}.

For each $\lambda\in\Lambda$, we apply~\eqref{e1049b} with $u$ replaced by
$u_\lambda\cdot a$ and $a$ replaced by $b$  to get the following
\begin{equation*}
\begin{split}
\widehat{\varphi}(u_\lambda\cdot (ab))
& =
\widehat{\varphi}((u_\lambda\cdot a)\cdot b)\\
& =
\widehat{\varphi}(b\cdot u_\lambda\cdot a)+
\widehat{\varphi}(\pi(u_\lambda\cdot a)\otimes b)
-S(b\pi(u_\lambda\cdot a))\\
& =
\widehat{\varphi}(b\cdot u_\lambda\cdot a)+\widehat{\varphi}((\pi(u_\lambda) a)\otimes b)
-S(b\pi(u_\lambda) a).
\end{split}
\end{equation*}
We thus have
\begin{equation}
\begin{aligned}
\label{e1128b}
&\widehat{\varphi}(u_\lambda\cdot (ab))
-
\widehat{\varphi}(u_\lambda\cdot (ba))\\
=&
\widehat{\varphi}(b\cdot u_\lambda\cdot a)-
\widehat{\varphi}(u_\lambda\cdot (ba))+
\widehat{\varphi}((\pi(u_\lambda) a)\otimes b)-S(b\pi(u_\lambda) a)\\
=&
\widehat{\varphi}((b\cdot u_\lambda-u_\lambda\cdot b)\cdot a)+
\widehat{\varphi}((\pi(u_\lambda) a)\otimes b)-S(b\pi(u_\lambda) a).
\end{aligned}
\end{equation}

On account of~\eqref{ad1}, we have
$\lim_{\lambda\in\Lambda}(b\cdot u_\lambda-u_\lambda\cdot b)=0$
and therefore
$\lim_{\lambda\in\Lambda}(b\cdot u_\lambda-u_\lambda\cdot b)\cdot a=0$,
which implies that
$\lim_{\lambda\in\Lambda}\widehat{\varphi}((b\cdot u_\lambda-u_\lambda\cdot b)\cdot a)=0$.
Since $\mathcal{U}$ refines the order filter on $\Lambda$, it follows that
$\lim_{\mathcal{U}}\widehat{\varphi}((b\cdot u_\lambda-u_\lambda\cdot b)\cdot a)=0$.

According to~\eqref{ad2}, we have
$\lim_{\lambda\in\Lambda}\pi(u_\lambda)a=a$.
Hence $$\lim_{\lambda\in\Lambda}(\pi(u_\lambda)a)\otimes b=a\otimes b\quad\mbox{and}\quad
\lim_{\lambda\in\Lambda}b\pi(u_\lambda) a=ba.$$
The continuity of both $\widehat{\varphi}$ and $S$ then gives
$$\lim_{\lambda\in\Lambda}\widehat{\varphi}((\pi(u_\lambda)a)\otimes b)=\widehat{\varphi}(a\otimes b)\quad\mbox{and}\quad
\lim_{\lambda\in\Lambda}S(b\pi(u_\lambda) a)=S(ba).$$
Since $\mathcal{U}$ refines the order filter on $\Lambda$, we conclude that
$$\lim_{\mathcal{U}}\widehat{\varphi}((\pi(u_\lambda)a)\otimes b)=\widehat{\varphi}(a\otimes b) \quad\mbox{and}\quad
\lim_{\mathcal{U}}S(b\pi(u_\lambda)a)=S(ba).$$

We now prove that
\begin{equation}\label{e1652}
\varphi(a,b)=T([a,b])+S(ba) \quad (a,b\in A).
\end{equation}
Indeed,
by taking the limit along $\mathcal{U}$ in~\eqref{e1128b} we arrive at
\begin{align*}
T([a,b]) & =T(ab)-T(ba)=
\lim_\mathcal{U}\widehat{\varphi}(u_\lambda\cdot (ab))-
\lim_\mathcal{U}\widehat{\varphi}(u_\lambda\cdot (ba)) \\
& =
\lim_\mathcal{U}(\widehat{\varphi}(u_\lambda\cdot (ab))-
\widehat{\varphi}(u_\lambda\cdot (ba))) \\
& =
\lim_\mathcal{U}\widehat{\varphi}((b\cdot u_\lambda-u_\lambda\cdot b)\cdot a)+
\lim_\mathcal{U}\widehat{\varphi}((\pi(u_\lambda) a)\otimes b) \\
& \quad  {}-\lim_{\mathcal{U}}S(b\pi(u_\lambda)a) \\
& = \widehat{\varphi}(a\otimes b)-S(ba)=\varphi(a,b)-S(ba).
\end{align*}

Define $\Phi\colon[A,A]\to X^{**}$ and $\Psi\colon A\to X$ by
\[
\Phi(a)=(T-\tfrac{1}{2}S)(a) \quad (a\in [A,A])
\]
and
\[
\Psi=\tfrac{1}{2} S.
\]
Note that, on account of~\eqref{e1652},
$T$ maps $[A,A]$ into $X$ and therefore
$\Phi$ does not map merely into $X^{**}$, but actually into $X$.
From~\eqref{e1652} we see that $\varphi(a,b)=\Phi([a,b])+\Psi(a\circ b)$ for all $a,b\in A$.
\end{proof}

\begin{theorem}\label{tab}
Let $A$ be an amenable Banach algebra with property $\mathbb{B}$,
let $X$ be a Banach space, and
let $\varphi\colon A\times A\to X$ be a continuous bilinear map satisfying the condition:
\begin{equation*}
a,b\in A, \ ab=ba=0 \ \Rightarrow \ \varphi(a,b)=0.
\end{equation*}
Then  there exist continuous linear operators
$\Phi\colon [A,A]\to X$ and $\Psi\colon A\to X$ such that
\begin{equation*}
\varphi(a,b)=\Phi([a,b])+\Psi(a\circ b)
\end{equation*}
for all $a$, $b\in A$.
\end{theorem}

\begin{proof}
A straightforward consequence of
Lemmas \ref{l1519} and \ref{l1520}.
\end{proof}

\begin{corollary}\label{cab}
If $A$ is an amenable Banach algebra with property $\mathbb{B}$, then $A$ is a zero Lie product determined Banach algebra.
\end{corollary}

\begin{proof}
Let $\varphi\colon A\times A\to\mathbb{C}$ be a continuous bilinear functional satisfying~\eqref{B}.
If $a,b\in A$ are such that $ab=ba=0$, then $[a,b]=0$ and therefore $\varphi(a,b)=0$. Consequently,
the functional $\varphi$ satisfies the condition in Theorem~\ref{tab}. Hence there exist continuous
linear functionals $\tau_1\colon [A,A]\to\mathbb{C}$ and $\tau_2\colon A\to\mathbb{C}$ such that
\begin{equation}\label{ecab1}
\varphi(a,b)=\tau_1([a,b])+\tau_2(a\circ b) \quad (a,b\in A).
\end{equation}
Of course, the functional $\tau_1$ extends to a continuous linear functional on $A$.
On the other hand, if $a\in A$, then $[a,a]=0$ and therefore $\varphi(a,a)=0$.
Hence $\varphi$ is skew-symmetric and~\eqref{ecab1} yields
\begin{equation}\label{ecab2}
\begin{aligned}
\varphi(a,b)&=-\varphi(b,a)=-\tau_1([b,a])-\tau_2(b\circ  a) \\
&=\tau_1([a,b])-\tau_2(a\circ b) \quad (a,b\in A).
\end{aligned}
\end{equation}
Adding~\eqref{ecab1} and~\eqref{ecab2}, we obtain
\[
\varphi(a,b)=\tau_1([a,b]) \quad (a,b\in A),
\]
which shows that $\varphi$ is of the form~\eqref{Bl}.
\end{proof}

Since the group algebra $L^1(G)$ has property $\mathbb{B}$ for each locally compact group $G$ and further it is amenable exactly
in the case when $G$ is amenable,  the following result follows.

\begin{theorem}
Let $G$ be an amenable locally compact group. Then the group algebra $L^1(G)$ is a zero Lie product determined Banach algebra.
\end{theorem}

\end{document}